\documentclass{article}
\usepackage{geometry}
\usepackage{authblk}
\usepackage{graphicx}     
\usepackage{amsmath,amssymb}
\usepackage{amsthm}
\usepackage{bm}
\usepackage{float}
\usepackage{siunitx}
\usepackage{subcaption}
\usepackage{tikz}

\newtheorem{theorem}{Theorem}[section]

\title{Error Estimation and Stopping Criteria for Krylov-Based Model Order Reduction in Acoustics} 

\author[1]{Siyang Hu}
\author[2]{Nick Wulbusch}
\author[2]{Alexey Chernov}
\author[3,1]{Tamara Bechtold}
\affil[1]{University of Rostock, Albert-Einstein-Str. 2, 18059 Rostock, Germany (e-mail: \{siyang.hu, tamara.bechtold\}@uni-rostock.de)}
\affil[2]{University of Oldenburg, Ammerländer Heerstraße 114 - 118, 26129 Oldenburg (Oldb), Germany, (e-mail: \{nick.wulbusch, alexey.chernov\}@uni-oldenburg.de)}
\affil[3]{Jade University of Applied Sciences, Friedrich-Paffrath-Str. 101, 26389 Wilhelmshaven, Germany (e-mail: tamara.bechtold@jade-hs.de).}
\date{}              
\begin{document}
\maketitle
\begin{abstract}              
Depending on the frequency range of interest, finite element-based modeling of acoustic problems leads to dynamical systems with very high dimensional state spaces. As these models can mostly be described with second order linear dynamical system with sparse matrices, mathematical model order reduction provides an interesting possibility to speed up the simulation process. 
In this work, we tackle the question of finding an optimal order for the reduced system, given a desired accuracy. To do so, we revisit a heuristic error estimator based on the difference of two reduced models from two consecutive Krylov iterations. We perform a mathematical analysis of the estimator and show that the difference of two consecutive reduced models does provide a sufficiently accurate estimation for the true model reduction error. This claim is supported by numerical experiments on two acoustic models. We briefly discuss its feasibility as a stopping criterion for Krylov-based model order reduction.\bigskip



\end{abstract}


\section{Introduction}
For acoustic simulations, frequency range of interest typically correlates with the human hearing range that goes up to 20 000~\si{\hertz}. For an accurate resolution of acoustic waves with the finite-element-method (FE method), empirical studies require the mesh size $h$ to be at least five times smaller than its wave length $\lambda$, i.e., $\lambda/h>10$ (\cite{ihlenburg1998finite}). Therefore, studies covering the entire human hearing range currently cannot be conducted by state-of-the-art computer-aided engineering tools (\cite{deckers2020case}).
Even FE analysis in mid-frequency range can be challenging due to the required mesh size. However, FE modeling of low to mid-frequency acoustic problems without complex damping results in a frequency independent, linear second order dynamical systems with typically very high dimensional but sparse matrices (\cite{deckers2020case}). This is an ideal scenario for the application of model order reduction (MOR) and can significantly reduce the required computational effort.

For general linear systems, MOR is well studied (\cite{antoulas2005approximation}). Application of MOR on acoustics and vibro-acoustics can be found in e.g., \cite{peters2014modal} using modal decomposition or \cite{puri2009reduced} using Krylov-based moment matching. For Krylov subspace-based MOR on structural dynamics and acoustics problems, a review is provided in \cite{hetmaniuk2012review}. One of the most recent works, \cite{aumann2023structured}, compares results of Krylov-based MOR with second order balanced truncation, applied to structural dynamics and acoustics problems. 

One important aspect of MOR is controlling the error. The most recent publication, \cite{feng2024posteriori}, surveys on different error estimates and bounds introduced in a series of previous works on parametric MOR. In \cite{feng2017some}, a residual-based error bounds was introduced, aiming at finding the optimal distribution of expansion points. Both parametric MOR and optimal distribution of expansion points are, however, currently outside the scope of this work. Our focus here is finding the reduce order model (ROM) with an optimal order that fulfills the desired accuracy. In this regard, \cite{panzer2013h} introduced $\mathcal{H}_2$ and $\mathcal{H}_{\infty}$ error bounds for Krylov-based MOR of systems with strictly dissipative realizations. Unfortunately, during our testings, a strictly dissipative realization could not always be achieved for general acoustic systems.

In this work, we take a deeper look at the error estimator based on the relative difference of ROMs generated by two consecutive Krylov iterations. The estimator was introduced in \cite{bechtold2004error} as an heuristic, and based on observations of purely thermal models. It is later implemented into 'Model Reduction inside Ansys` \cite{bechtold2007fast}, which has become industrial standard. We show that mathematically, the consecutive error indeed provides a good approximation for the actual error introduced by MOR. We evaluate the performance of the estimator on two acoustic models and discuss its suitability as a stopping criterion. 

\section{Mathematical Modeling of Acoustic Systems}
Time-harmonic phenomena in acoustics are frequently modeled by the Helmholtz equation enhanced by appropriate boundary conditions. Specifically,
consider a sufficiently regular and bounded domain $\Omega \subset \mathbb{R}^3$, and a function $f:\Omega\rightarrow \mathbb{R}$, $f\in L^2(\Omega)$ representing volume sources. Then the complex pressure amplitude $p=p(\bm{x})$ corresponding to a fixed wave number $k = 2\pi/\lambda$ satisfies
\begin{align}
    \begin{split}
        -\Delta p - k^2 p &= f \quad \mbox{in } \Omega,\\
        \partial_n p &= g \quad \mbox{on } \Gamma_N,\\
        \partial_n p + i k p &= 0 \quad \mbox{on } \Gamma_R,
    \end{split}\label{eq:helmholtz}
\end{align}
where $\Gamma_N\cup\Gamma_R=\partial\Omega$ denote the acoustically rigid and impedance parts of the boundary, the symbol $\partial_n$ denotes the normal derivative pointing outside $\Omega$, $g$ describes the excitation of a source on the boundary, and $i$ is the imaginary unit.

In this work, we focus on solving the boundary value problem \eqref{eq:helmholtz} using the finite element method (FEM). After applying spatial discretization, we obtain
\begin{equation}
\Sigma:=
\begin{cases}
    (-k^2\bm{M}+ik\bm{D}+\bm{K})\bm{{p}} = \bm{B{u}}\\
    \bm{y} = \bm{Cp}.
    \label{eq:sys}
\end{cases}
\end{equation}
Here, $\bm{{p}}\in\mathbb{C}^n$ is the state vector, containing the complex nodal amplitudes. $\bm{M},\bm{D},\bm{K}\in\mathbb{R}^{n\times n}$ are the mass, damping and stiffness matrix, respectively. The complex load $\bm{{u}}\in\mathbb{C}^{p_i}$ is scaled and distributed onto the nodes via the input matrix $\bm{B}\in\mathbb{R}^{n\times p_i}$ and the output vector $\bm{y}\in\mathbb{R}^{p_o}$ is gathered via a user-defined output matrix $\bm{C}\in\mathbb{R}^{p_o\times n}$. Solving acoustic problems typically requires solving a series of harmonic analysis \eqref{eq:sys}, i.e., for each frequency / wave number of interest.

\section{Krylov-based Model Order Reduction}
The goal of general MOR is to find a significantly lower dimensional surrogate for $\Sigma$ from \eqref{eq:sys}
\begin{equation}
    \Sigma_r := \begin{cases}
        (-k^2\bm{M}_r+ik\bm{D}_r+\bm{K}_r)\bm{{p}}_r = \bm{B_r{u}}\\
        \bm{y} = \bm{C}_r\bm{p}_r,
    \end{cases}\label{eq:sys_r}
\end{equation}
which approximates the behavior of the original system up to a defined error tolerance. Typically, the approximation is achieved via projection onto a suitable subspace using appropriate matrices $\bm{W}, \bm{V}\in\mathbb{R}^{n\times r}, r\ll n$, i.e.,
\begin{gather}
    \bm{p} \approx \bm{Vp}_r, \\
    \{\bm{M}_r,\bm{D}_r,\bm{K}_r\} = \bm{W^T}\{\bm{M},\bm{D},\bm{K}\}\bm{V},\\
    \bm{B}_r = \bm{W^TB}, \quad \bm{C}_r = \bm{CV}.
\end{gather}
The input $\bm{u}$ and the output $\bm{y}$ are untouched from \eqref{eq:sys}. When we choose $\bm{W} = \bm{V}$, then the projection is called one-sided. 

In Krylov-based MOR, the columns of matrices $\bm{W}$ and $\bm{V}$ are typically chosen as an orthonormal basis of 
the input and output Krylov subspaces, respectively. For second order systems, \cite{salimbahrami2006order} introduced the second order Krylov subspaces, which are defined as follows:
\begin{equation}
    \mathcal{K}_r(\bm{A_1},\bm{A_2},\bm{P_0}) = {\rm colspan}\{\bm{P_0},\bm{P_1},\dots,\bm{P_{r-1}}\},
\end{equation}
where
\begin{equation}
    \begin{cases}
        \bm{P_1} = \bm{A_1P_0}\\
        \bm{P_i} = \bm{A_1P_{i-1}}+\bm{A_2P_{i-2}},\quad i = 2,3,\dots.
    \end{cases}    
\end{equation}
With this definition, the second order input Krylov subspace corresponding to \eqref{eq:sys} about a chosen expansion point $k_0$ is defined as
\begin{equation}
    \mathcal{K}_{r_1}(-\bm{\tilde{K}}^{-1}\bm{\tilde{D}},-\bm{\tilde{K}}^{-1}\bm{M},-\bm{\tilde{K}}^{-1}\bm{B})
\end{equation}
and the second order output Krylov subspace of \eqref{eq:sys} is defined as 
\begin{equation}
    \mathcal{K}_{r_2}(-\bm{\tilde{K}}^{-T}\bm{\tilde{D}}^T,-\bm{\tilde{K}}^{-T}\bm{M}^T,-\bm{\tilde{K}}^{-T}\bm{C}^T),
\end{equation}
where, $\bm{\tilde{K}} = s_0^2\bm{M}+s_0\bm{D}+\bm{K}$ and $\bm{\tilde{D}} = 2s_0\bm{M}+\bm{D}$ are the shifted system matrices for the fixed parameter value $s_0=ik_0$. It is shown in \cite{lohmann2005reduction}, that with this choice of projection matrices, the first $r_1+r_2$ moments of the full order model (FOM) $\Sigma$ and the ROM $\Sigma_r$ are matched, which are the negative coefficients of the Taylor series expansion of both system's transfer functions
\begin{equation}
    \bm{G}(s) = \bm{C}[(s-s_0)^2\bm{M}+(s-s_0)\bm{\tilde{D}}+\bm{\tilde{K}}]^{-1}\bm{B}, \label{eq:g(s)}
\end{equation}
and
\begin{equation}
    \bm{G}_r(s) = \bm{C}_r[(s-s_0)^2\bm{M}_r+(s-s_0)\bm{\tilde{D}}_r+\bm{\tilde{K}}_r]^{-1}\bm{B}_r. \label{eq:gr(s)}
\end{equation}
Here, $\bm{\tilde{D}}_r$ and $\bm{\tilde{K}}_r$ are, again the shifted system matrices of the ROM, defined analogously to the FOM, and $s$ is a complex parameter sufficiently close to $s_0$, so that $\bm{G}(s)$ and $\bm{G}_r(s)$ are well-defined. For acoustic systems, notice that $G(s_0)$ is well-defined for $k>0$ according to \cite[Theorem 2.27]{ihlenburg1998finite} for sufficiently small mesh size $h$. We emphasize that \cite{lohmann2005reduction} only considered single-input-single-output systems (i.e., $p_i=p_o=1$), however, the moment matching ability also extends to single-input-multiple-output and 
multiple-input-multiple-output systems due to the superposition property of linear systems (\cite{benner2008using}).

Based on their findings, the authors introduced the second order Arnoldi algorithm (SOAR), which has become state of the art for MOR of second order systems. SOAR is a one-sided approach and, therefore, only matches the first $r_1$ moments when e.g., the input Krylov subspace is used.

\section{Estimation of Approximation Error}
Let us consider the relative frequency-response error
\begin{equation}
    E_r(s) = \frac{||\bm{G}(s)-\bm{G}_r(s)||}{||\bm{G}(s)||}, \label{eq:freq_error}
\end{equation}
where $\bm{G}(s)$ and $\bm{G}_r(s)$ are the transfer functions of the FOM and the reduced order model (ROM) of order $r$. In \cite{bechtold2004error}, the authors suggest estimating the frequency error $E_r$ with the relative frequency-response error between two ROMs arising from two consecutive Krylov iteration  
\begin{equation} 
    \hat{E}_r(s) = \frac{||\bm{G}_{r+1}(s)-\bm{G}_r(s)||}{||\bm{G}_r(s)||}. \label{eq:freq_error_estim}
\end{equation}
For their considered thermal models, the authors observed a good match between the relative error and the estimator. Therefore, we will have a closer look at the estimator to evaluate its suitability for acoustic problems.
As an alternative to \eqref{eq:freq_error_estim}, we consider the rescaled version
    \begin{equation}
    \tilde{E}_r(s) = \frac{||\bm{G}_{r+1}(s)-\bm{G}_r(s)||}{||\bm{G}_{r+1}(s)||}. 
    \label{eq:error_estim_2}
\end{equation}
The following Theorems addresses the accuracy of the above estimators.

\begin{theorem}\label{thm:abs_error}
    Let $s_0=ik_0$ and $s$ be sufficiently close to $s_0$ so that $\bm{G}(s)$ is well-defined according to \eqref{eq:g(s)}.
    Let $\bm{G}_r$ be the transfer function of the reduced order model of order $r$ according to \eqref{eq:gr(s)} obtained after $r$ iterations of a Krylov subspace-based MOR method and suppose that $\bm{G}_r(s)$ is well-defined, i.e., the inverse in \eqref{eq:gr(s)} exists.
    Then $||\bm{G}_{r+1}(s)-\bm{G}_r(s)||$ is an estimator for the absolute error $||\bm{G}(s)-\bm{G}_r(s)||$ and there holds
    \begin{equation*}
         ||\bm{G}(s)-\bm{G}_r(s)|| = ||\bm{G}_{r+1}(s)-\bm{G}_r(s)|| + \mathcal{O}((s-s_0)^{r+1}).
    \end{equation*}
\end{theorem}
\begin{proof} 
The formal power series expansions for $\bm{G}(s)$ and $\bm{G}_r(s)$ read
    \begin{equation}
        \bm{G}(s) = \sum_{\ell=0}^\infty - \bm{m}_\ell (s-s_0)^\ell,
    \end{equation}
    and 
    \begin{equation}
        \bm{G}_r(s) = \sum_{\ell=0}^\infty - \bm{m}_{\ell}^{(r)} (s-s_0)^\ell,
    \end{equation}
    with $\bm{m}_\ell^{(r)}$ being the moments of the of the reduced system after $r$ iterations.
    Due to the moment-matching properties of Krylov-based MOR \cite[Theorem 4.1]{bai2005dimension}, we have $\bm{m}_\ell=\bm{m}_{\ell}^{(r)}$, $\ell=0,\dots,r-1$ and therefore there holds    
    \begin{align}
    \begin{split}
        \bigg| ||\bm{G}&(s)-\bm{G}_r(s)||  - ||\bm{G}_{r+1}(s)-\bm{G}_r(s)||\bigg|\\
        &\leq ||\bm{G}(s)-\bm{G}_{r+1}(s)||\\
        &= \bigg|\bigg| \sum_{\ell=r+1}^\infty (\bm{m}_\ell - \bm{m}_{\ell}^{(r+1)}) (s-s_0)^\ell\bigg|\bigg|\\
        &= \mathcal{O}((s-s_0)^{r+1}).
    \end{split}\label{eq:abs_error_comp}
    \end{align}
\end{proof}    
\begin{theorem} Suppose $\bm{G}(s_0) \neq 0$ and $\bm{G}_q(s_0) \neq 0$ for $q=r$ and $r+1$.
Under assumptions of Theorem \ref{thm:abs_error} both $\hat{E}_r$ and $\tilde{E}_r$ are estimators for the true relative approximation error $E_r$ of a ROM generated by moment-matching-based MOR about a given expansion point $s_0$. Moreover, there holds
\begin{equation}\label{eq:rel_err_1}
    E_r(s) = \hat{E}_r(s) + \mathcal{O}((s-s_0)^{r+1}),
\end{equation}
\begin{equation}\label{eq:rel_err_2}
    E_r(s) = \tilde{E}_r(s) + \mathcal{O}((s-s_0)^{r+1}).
\end{equation}
\end{theorem}
\begin{proof}
Bound \eqref{eq:abs_error_comp} implies
\begin{equation}
    \bigg|E_r(s) - \frac{||\bm{G}_{r+1}(s) - \bm{G}_r(s)||}{||\bm{G}(s)||} \bigg| = \mathcal{O}((s-s_0)^{r+1})
\end{equation}
and therefore, for $q=r$ and $r+1$
\begin{equation}
    \begin{split}
            \bigg|E_r(s) &- \frac{||\bm{G}_{r+1}(s) - \bm{G}_r(s)||}{||\bm{G}_q(s)||} \bigg| \\
            &\leq \frac{||\bm{G}(s)-\bm{G}_q(s)||}{||\bm{G}(s)||} \frac{||\bm{G}_{r+1}(s)-\bm{G}_r(s)||}{||\bm{G}_q(s)||}\\ 
            &+ \mathcal{O}((s-s_0)^{r+1})
    \end{split}
\end{equation}
Since $\bm{G}(s)$ is continuous and $\bm{G}(s_0) \neq 0$ we have 
\[
\frac{||\bm{G}(s)-\bm{G}_q(s)||}{||\bm{G}(s)||} = \mathcal{O}((s-s_0)^{q}).
\]
Analogously, since $\bm{G}_q(s)$ is continuous and $\bm{G}_q(s_0) \neq 0$ 
\[
 \frac{||\bm{G}_{r+1}(s)-\bm{G}_r(s)||}{||\bm{G}_q(s)||} = \mathcal{O}((s-s_0)^{r}).
\]
Since $q \geq r \geq 1$ this implies
\begin{equation}
    \bigg|E_r(s) - \frac{||\bm{G}_{r+1}(s) - \bm{G}_r(s)||}{||\bm{G}_q(s)||} \bigg| = \mathcal{O}((s-s_0)^{r+1})
\end{equation}
and therefore, \eqref{eq:rel_err_1} with $q=r$ and \eqref{eq:rel_err_2} with $q=r+1$.
\end{proof}

\section{Numerical Experiments}
In this section, we perform SOAR on two acoustic models. We compare error estimator \eqref{eq:error_estim_2} and the estimator for the absolute error (from Theorem \ref{thm:abs_error}) to their respective counterparts. For these experiments, we do not consider any volume sources and therefore, we set the right hand side of the Helmholtz equation in \eqref{eq:helmholtz} to $f=0$. The first experiment is carried out in Matlab, with our own implementation of the model and MOR. For the second experiment, the model is implemented in  Ansys\textsuperscript{\textregistered} Mechanical Enterprise Academic Research, Release 2023.2 (\cite{ansys}) and MOR is performed with the toolbox Model Reduction inside Ansys (\cite{bechtold2007fast}). 

\subsection{2D Helmholtz problem}
For the first experiment, we consider a simple two-dimensional Helmholtz problem on the unit square $\Omega=[0,1]^2$, as illustrated in Figure \ref{fig:2dHelmholtz}. The problem is described with \eqref{eq:helmholtz}, where we set $g=10i$. For finite-element-analysis, the domain is discretized with triangular elements in the size of $~2^{-8}$. The output of the system is defined as the acoustic pressure measured at the points marked in Figure \ref{fig:2dHelmholtz}.

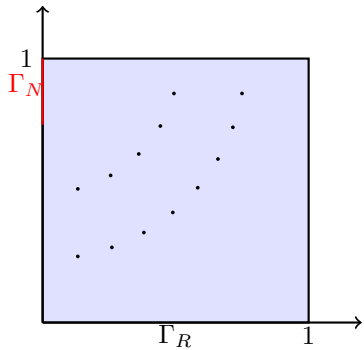
\begin{figure}[H]
    \centering
    \begin{tikzpicture}[scale=0.7]    
        \coordinate (A1) at (0cm,0cm);
                    \coordinate (A2) at (5cm,0cm);
                   
                    \coordinate (A3) at (5cm,5cm);
                    \coordinate (A4) at (0cm,5cm);
                   
                    \coordinate (A5) at (0cm,3.75cm);
                   
                    \fill[blue!60, opacity=0.2] (A1) -- (A2) -- (A3) -- (A4) -- cycle;
                    \draw[thick] (A1) -- (A2) -- (A3) -- (A4) -- (A1);
                   
                    \draw[thick,arrows=->] (A1) -- (6cm,0) node[right] {};
                    \draw[thick,arrows=->] (A1) -- (0,6cm) node[left] {};
                   
                    \node at (5cm,-0.25cm) {$1$};
                    \node at (-0.3cm,5cm) {$1$};
                   
                    \draw[thick,red] (A4) -- (A5);
                    \node at (-0.3cm,4.5cm) {\textcolor{red}{$\Gamma_{N}$}};
                    \node at (2.5cm,-0.25cm) {\textcolor{black}{$\Gamma_{R}$}};

                    \foreach \x in {1,...,5}
                    {
                                    \fill ({5*cos(-\x*15)*0.5111cm},{5*sin(-\x*15)*0.5111cm+5cm}) circle (1pt);
                    }
                    \foreach \x in {1,...,8}
                    {
                                    \fill ({5*cos(-\x*10)*0.7611cm},{5*sin(-\x*10)*0.7611cm+5cm}) circle (1pt);
                    }
    \end{tikzpicture}
    \caption{Domain for model problem. Points illustrate measurements points at which the error is computed in the numerical experiments.}
    \label{fig:2dHelmholtz}
\end{figure}

For MOR, we have chosen multiple expansion points at $k = \{20,60,100\}$. The projection matrix is constructed in the following way: First one column is added based on the first expansion point (20), the second one on the second expansion point (60) then the third (100). Then again for the first etc.. For example: The ROM of dimension 40 would have 14 columns from the first expansion point (20), while the expansion points 60 and 100 both contribute 13 columns.

\begin{figure}[H]
	\centering
	\includegraphics[width=0.98\linewidth]{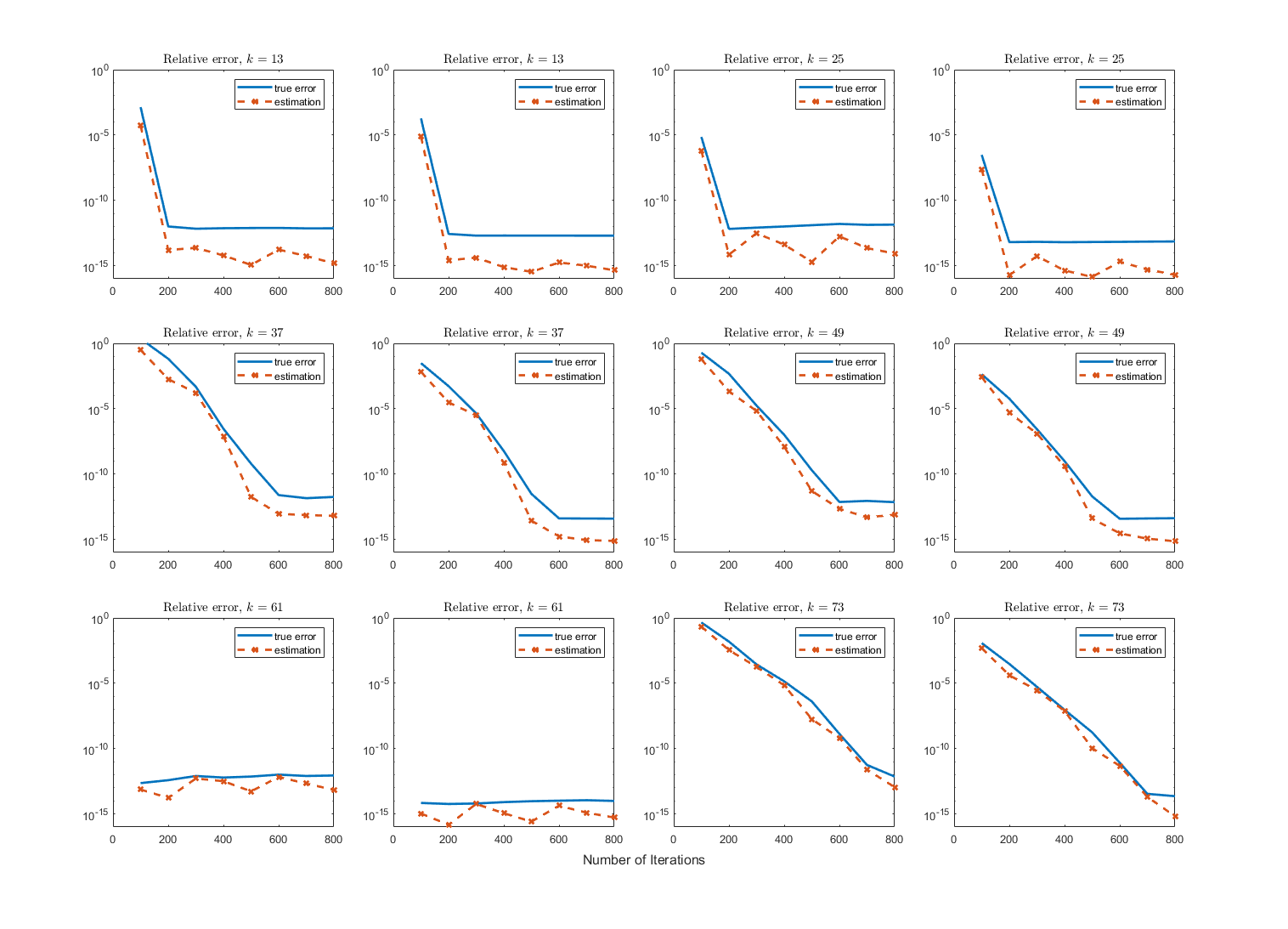}
	\caption{True error vs error estimator (sup-norm) for the 2D Helmholtz problem and ROMs of dimensions every 100 iterations up to 800.}
	\label{fig:2d_error_v_estim}
\end{figure}

In Figure \ref{fig:2d_error_v_estim}, we present the error in dependence of the number of iterations, i.e., the dimension of the ROM (for some fixed wave number $k$). We plotted the true MOR error and the estimator every 100 iterations up to a total of 800 iterations. We observe convergence to machine precision over the course of the Arnoldi iterations. The number of iterations required for convergence mainly depends on the distance between the considered wave number and the chosen expansion points.

\subsection{3D Helmholtz Problem: Speaker Plate}
For the second experiment, we consider a more realistic test-case. It consists of a speaker and a plate in a three-dimensional space. We set the computational to a sphere enclosing the speaker-plate setup, as depicted in Figure \ref{fig:3dHelmholtz}.

\begin{figure}[H]
    \centering
    \includegraphics[width=0.4\linewidth]{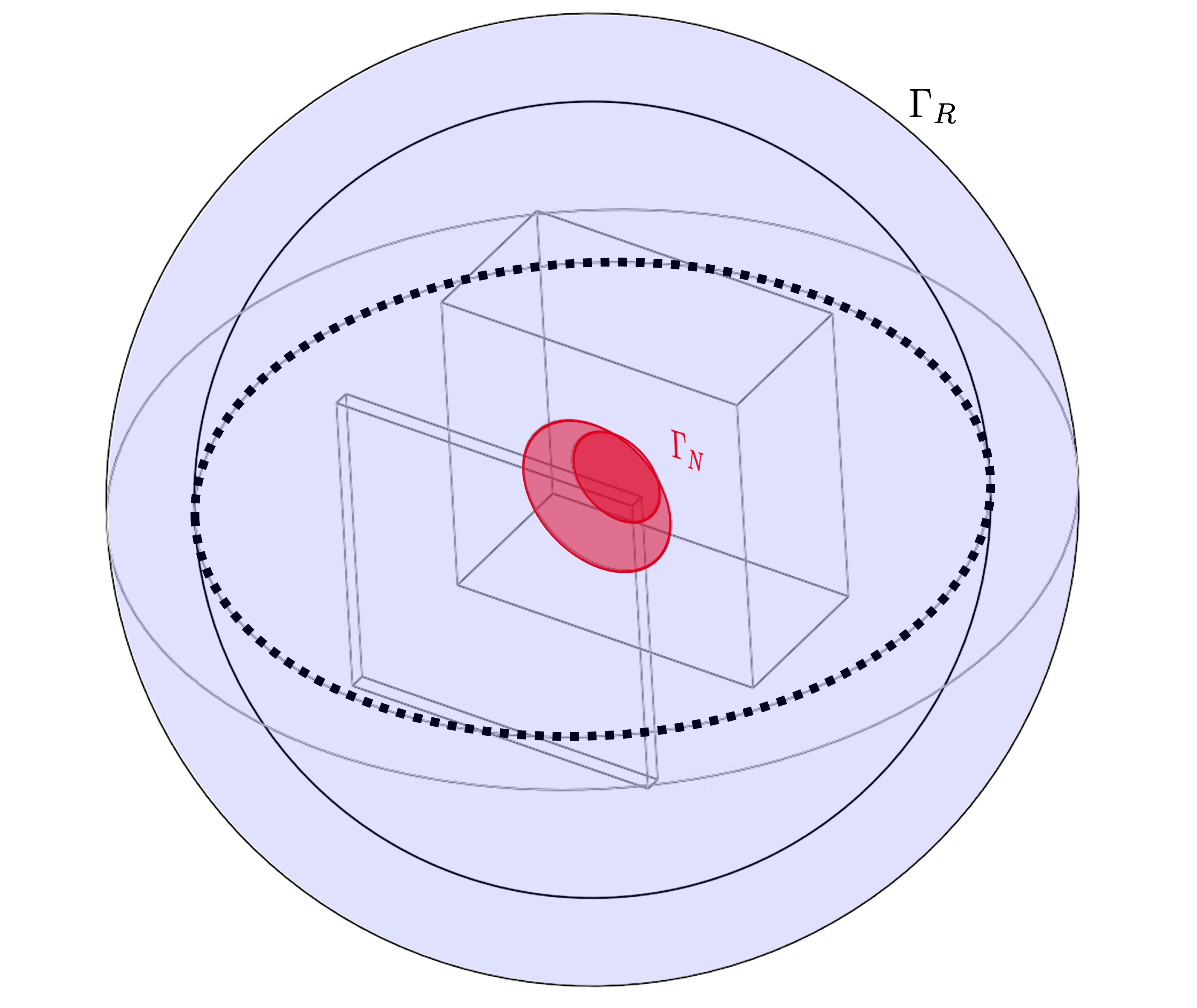}
    \caption{Domain for the 3D Helmholtz problem involving a speaker and a plate. The pointed line illustrates measurement points (all nodes on the line) at which the error is computed in the numerical experiments.}
    \label{fig:3dHelmholtz}
\end{figure}

Both the plate and the speaker box are square shaped at $0.1\times 0.1$~\si{\square\metre} and have their respective thicknesses valued at 0.005~\si{\metre} and 0.05~\si{\metre}. The inner and outer diameter of the speaker unit are 0.3~\si{\metre} and 0.5~\si{\metre} with the depth of the membrane being 0.01~\si{\metre}. We set the surface of the speaker unit as $\Gamma_N$ and define $g$ in such a way that it emits sound waves at a velocity of 0.001~\si{\metre\per\second}. The diameter of the inner sphere, on which the acoustic pressure is measured, is set to 0.226~\si{\metre}. It is encapsulated with a second sphere of 0.276~\si{\metre} diameter on top of which we have defined the absorption boundary condition. Note that we treat both the speaker and the plate as infinitely rigid and therefore represent perfect reflectors. The domain is discretized with tetrahedral elements with the element size set to 0.01~\si{\metre}.

\begin{figure}[h]
	\centering
	\begin{subfigure}[b]{0.64\textwidth}
		\centering
		\includegraphics[width=\linewidth]{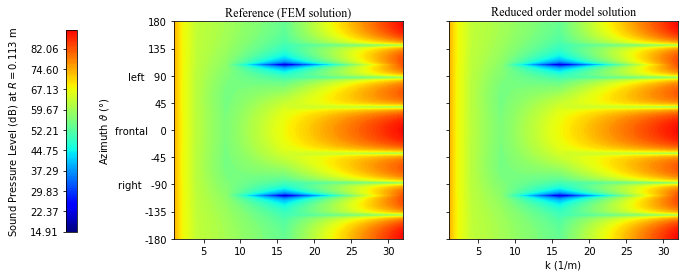}
	\end{subfigure}
	\hspace{7pt}
	\begin{subfigure}[b]{0.295\textwidth}
		\centering
		\includegraphics[width=\linewidth]{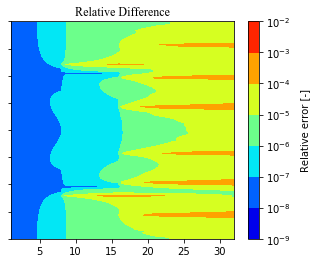}
	\end{subfigure}
	\caption{Comparison of FOM versus ROM plotted over different wave numbers $k$ and their relative difference.}
	\label{fig:3d_results}
\end{figure}

For MOR, we have chosen two expansion points at $k = \{2,18\}$~\si{\per\metre}. During this experiment, they are added sequentially, i.e., the first 400 columns of the projection matrix corresponds to $k=2~\si{\per\metre}$, the remaining columns to $k=18~\si{\per\metre}$. This results in a ROM of dimension $r=800$. A comparison of the FOM result versus the ROM results (sound pressure level at the output nodes) is shown shown in Figure \ref{fig:3d_results}. In Figure \ref{fig:3d_error_v_estim}, we present the error of the ROM over the ROM's dimensions. Similar to the two-dimensional case, the error converges to a minimum, valued at $~10^{-7}$. Here however, the estimator underestimates the error significantly upon convergence of the true error. For higher wave numbers, one can clearly see the impact of the additional expansion point, which is included after 400 iterations.
\begin{figure}[H]
	\centering
	\includegraphics[width=\linewidth,trim=50 40 50 0,clip=true]{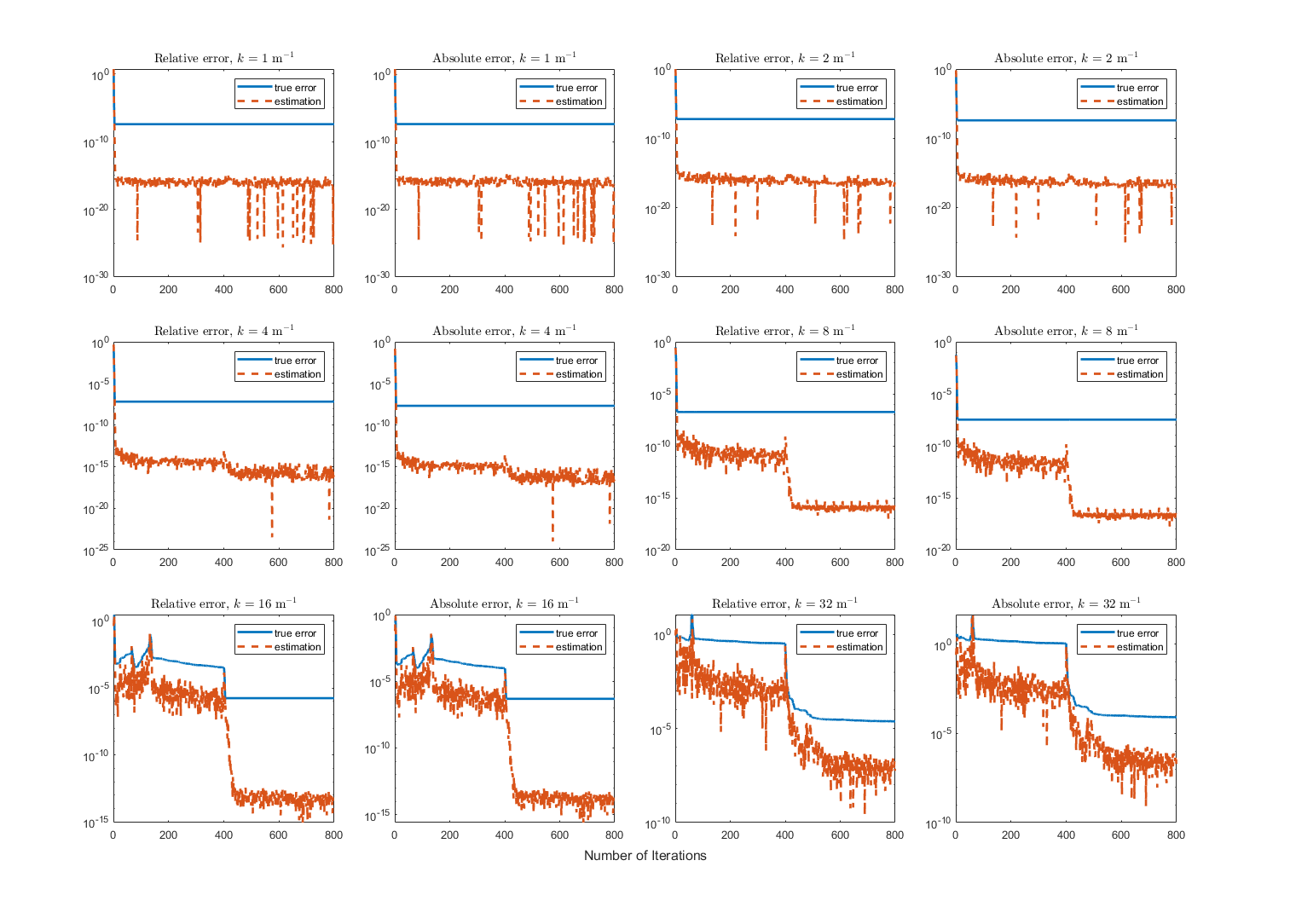}
	\caption{True error vs error estimator (2-norm) for the 3D Helmholtz problem and ROMs of dimensions up to 800.}
	\label{fig:3d_error_v_estim}
\end{figure}
\section{Conclusion \& Outlook}

In this work, we analysed the heuristic error estimator based on the difference between two ROMs generated by two consecutive iterations during Krylov-based MOR. We provided a mathematical foundation for the estimator and tested its performance on two acoustic models. While it performed well for the simple two-dimensional model, the estimator underestimated the error significantly for the more realistic three-dimensional model. We may account this to the worse condition number of the system matrices in the 3D case. However, the underestimation only occurred, when the true error has converged to its minimal value. Therefore, the estimation can still be used as a stopping criterion, e.g., stop the MOR algorithm, once the estimation has converged to its minimum value. However, due the oscillation of the estimator, one may consider some kind of low pass filtering, e.g., moving average. 

The evaluation of the MOR error has also shown, that the distances between the wave number of interest and the chosen expansion points significantly influence the MOR error. Therefore, future works may consider also taking the optimal distribution of expansion points into account.

\end{document}